\documentclass[12pt]{article}
\usepackage{amssymb}
\usepackage{amsmath}
\usepackage{graphs}
\usepackage{amssymb}
\usepackage{amsthm}
\usepackage{url}
\usepackage{graphicx}

\theoremstyle{plain}
\newtheorem{definition}{Definition}

\newtheorem{theorem}[definition]{Theorem}
\newtheorem{conjecture}[definition]{Conjecture}

\newtheorem{problem}[definition]{Problem}

\newtheorem{example}{Example}

\newcommand{\lef}{\mathcal{L}}
\newcommand{\pre}{\mathcal{P}}
\newcommand{\n}{\mathcal{N}}
\newcommand{\ri}{\mathcal{R}}
\newcommand{\ti}{\mathcal{T}} 

\newcommand{\gfr}{G_F^{SR}}
\newcommand{\gfl}{G_F^{SL}}
\newcommand{\plus}{+_{\ell}}

\begin{document}
\title{Pirates and Treasure}
\author{Fraser Stewart\\
\small{Xi'An Jiaotong University, Department of Mathematics and Statistics, Xi'An, China}\\
\texttt{fraseridstewart@gmail.com}}
\date{}

\maketitle{}

\begin{abstract}
In this paper we introduce a new game; in this game there are two players, who play as rival pirate gangs.  The goal is to gather more treasure than your rival.  The game is played on a graph and a player gathers treasure by moving to an unvisited vertex.  At the end of the game, the player with the most treasure wins.  

We will show that this game is NP-Hard, and we will also look at the structure of this game under the disjunctive sum.  We will show that there are cases where this game behaves like a normal play game, and cases where it behaves like a mis\`ere play game.  We then leave an open problem about scoring play games in general.
\end{abstract}

\section{Introduction}

Coin-sliding games have been studied for many years, one of the best known is the game Geography.  This is a simple game that parents often tell their children to play during long car journeys.  The idea is that a person says the name of a country, and the next person must name a country whose first letter is the same as the last letter of the country just named.  For example, Britain, Norway, Yugoslavia, America, Argentina, Australia and so on.

The generalised version of this game is played on a directed graph, and players take it in turns to move a coin to a previously unvisited neighbour.  The game ends when a player cannot move, and the last player to move is the winner.

Another game, simply titled ``The Coin-Sliding Game'', was introduced in David Moews' paper \cite{dm}.  This is a game where the players have coins of various values, that are placed on a vertical strip.  The player then chooses to either move a coin down the strip, or remove one of his opponents coins.  

The players collect coins that are removed from the strip, either by sliding them off it, or from removing them.  At the end of the game the players add the values of the coins that they have collected, and the player who has the most wins.  In his paper, David gave the solution to this game.

Our game is played on an undirected graph, where players have multiple coins placed on various vertices of the graph.  These coins represent their pirate ships.  The remaining vertices are given a numerical value, and players gather points (treasure) by moving onto those vertices.  The player who gathers the most points (treasure) wins.

\vspace{0.5cm}

The formal rules are given as follows:

\begin{enumerate}
\item{The game is played on a finite simple graph, defined arbitrarily before the game begins.  Left has $n$ ships, and Right has $m$ ship.}
\item{Each ship has a pre-defined starting vertex.}
\item{Every node is numbered to indicate how much treasure there is at that node, the players starting vertices are not numbered.}
\item{On a player's turn he moves to an adjacent, unvisited vertex.  The number of points he gets, corresponds to the number on that vertex. A player may not move a previously visited vertex, including the starting vertices.}
\item{The game ends when it is a player's turn and he is not adjacent to an unvisited vertex.}
\item{The player who gathers the most treasure wins.}
\end{enumerate}

In this paper, first we will be examining the complexity of this game, and then showing that there are variations of it which are comparable with normal and mis\`ere play.

\subsection{Scoring Play Combinatorial Game Theory}

Scoring play combinatorial games have not been studied anywhere near as much as their normal, and misere play counter parts.  The first papers were written by Milnor and Hanner \cite{oh, jm}.  There are also a papers by Jeff Ettinger \cite{je1, je2}, neither of which were ever published.  Will Johnson did some follow up work subsequently \cite{wj}.

All of them studied well-tempered scoring play games, that is games where the game always lasts a fixed number of moves.  However, in 2011 Fraser Stewart introduced the most general theory for scoring play games \cite{fs}.  This work was done entirely independently of Milnor, Hanner and Johnson, and is based on the theories of Elwyn Berlkeamp, John Conway and Richard Guy \cite{ww, onag}.

The idea behind this theory is very simple, consider the game tree given in figure \ref{gt}.

\begin{figure}[htb]
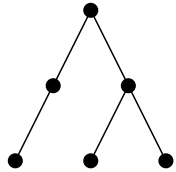


\begin{center}
\begin{graph}(2,2)

\roundnode{2}(0,0)\roundnode{3}(0.5,1)\roundnode{4}(1,2)
\roundnode{5}(1.5,1)\roundnode{6}(1,0)\roundnode{7}(2,0)

\edge{2}{3}\edge{3}{4}
\edge{4}{5}\edge{5}{6}\edge{5}{7}

\end{graph}
\end{center}

\caption{A typical game tree.}\label{gt}
\end{figure}

On a typical game tree as shown in figure \ref{gt}, the nodes represent the positions of a game, and edges represent possible moves for both players from those positions.  Left sloping edges are Left's moves, and right sloping edges are Right's moves.

A scoring play game tree is exactly the same, but for one difference, the nodes now have numbers on them which represent the score associated with that position.  The score is the difference between Left's total points, and Right's total points, at that point in the game.

\begin{figure}[htb]
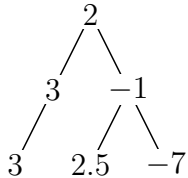


\begin{center}
\begin{graph}(2,2)

\roundnode{2}(0,0)\roundnode{3}(0.5,1)\roundnode{4}(1,2)
\roundnode{5}(1.5,1)\roundnode{6}(1,0)\roundnode{7}(2,0)

\edge{2}{3}\edge{3}{4}
\edge{4}{5}\edge{5}{6}\edge{5}{7}

\nodetext{2}{3}\nodetext{3}{3}\nodetext{4}{2}
\nodetext{5}{$-1$}\nodetext{6}{2.5}\nodetext{7}{$-7$}

\end{graph}
\end{center}

\caption{A scoring play game tree.}
\end{figure}

Formally scoring play games are defined as follows.

\begin{definition} A scoring play game $G=\{G^L|G^S|G^R\}$, where $G^L$ and $G^R$ are sets of games and $G^S\in\mathbb{R}$, the base case for the recursion is any game $G$ where $G^L=G^R=\emptyset$.\\

\noindent $G^L=\{\hbox{All games that Left can move to from } G\}$\\
\noindent $G^R=\{\hbox{All games that Right can move to from } G\}$,\\

and for all $G$ there is an $S=(P,Q)$ where $P$ and $Q$ are the number of points that Left and Right have on $G$ respectively.  Then $G^S=P-Q$, and for all $g^L\in G^L$, $g^R\in G^R$, there is a $p^L,p^R\in\mathbb{R}$ such that $g^{LS}=G^S+p^L$ and $g^{RS}=G^S+p^R$.

\end{definition}

By convention, we will take $G^S$ to be 0, unless stated otherwise.  This is simply to give games a ``default'' setting, i.e. if we don't know what $G^S$ is then it is natural to simply let it be 0.  We also write $\{.|G^S|.\}$ as $G^S$, e.g. $\{\{.|3|.\}|4|\{.|2|.\}\}$ would be written as $\{3|4|2\}$.  This simply for convenance and ease of reading.

For these games we also need the idea of a ``final score''.  That is the best possible score that both players can get when they move first.  Formally, this is defined as follows.

\begin{definition}  We define the following:

\begin{itemize}
\item{$\gfl$ is called the Left final score, and is the maximum score --when Left moves first on $G$-- at a terminal position on the game tree of $G$, if both Left and Right play perfectly.}
\item{$\gfr$ is called the Right final score, and is the minimum score --when Right moves first on $G$-- at a terminal position on the game tree of $G$, if both Left and Right play perfectly.}
\end{itemize}

\end{definition}

Our game Pirates and Treasure clearly falls under this theory, so we will be referencing it --and using it-- throughout the paper.

The paper ``Scoring Play Combinatorial Game Theory'' \cite{fs} discusses the structure of these games under the disjunctive sum, which is defined below.  In this paper it is shown that these games do not form a group, there is no non-trivial identity, and almost no games that can be compared in the usual sense.  However, these games are partially ordered under the disjunctive sum, and do form equivalence classes with a canonical form.  The games can also be reduced using the usual rules of domination and reversibility.

\begin{definition}  The disjunctive sum is defined as follows:
$$G\plus H=\{G^L\plus H,G\plus H^L|G^S+H^S|G^R\plus H,G\plus H^R\},$$
\noindent where $G^S+H^S$ is the normal addition of two real numbers.
\end{definition}

We abuse notation by letting $G^L$ and $G^R$ represent the set of options and the individual options themselves.  The reader will also notice that we have used $\plus$ and $+$.  This is to distinguish between the addition of games (the disjunctive sum), and the addition of scores.

\begin{definition} We define the following:
\begin{itemize}
\item{$-G=\{-G^R|-G^S|-G^L\}$.}
\item{For any two games $G$ and $H$, $G=H$ if $G\plus X$ has the same outcome as $H\plus X$ for all games $X$.}
\item{For any two games $G$ and $H$, $G\geq H$ if $H\plus X\in O$ implies $G\plus X\in O$, where $O=L_\geq$, $R_\geq$, $L_>$ or $R_>$, for all games $X$.}
\item{For any two games $G$ and $H$, $G\leq H$ if $H\plus X\in O$ implies $G\plus X\in O$, where $O=L_\leq$, $R_\leq$, $L_<$ or $R_<$, for all games $X$.}
\item{$G\cong H$ means $G$ and $H$ have identical game trees.}
\item{$G\approx H$ means $G$ and $H$ have the same outcome.}
\end{itemize}
\end{definition}  

The reason that we need this, is because our game naturally splits up into multiple smaller components that are played independently of one another.  So, the disjunctive sum is the natural operator to use when analysing this game.

Finally we need to define the outcome classes of games.  Before we can define what the outcome classes are precisely, we first need the following definition.

\begin{definition}
\item{$L_>=\{G|\gfl>0\}$, $L_<=\{G|\gfl<0\}$, $L_= =\{G|\gfl=0\}$.}
\item{$R_>=\{G|\gfr>0\}$, $R_<=\{G|\gfr<0\}$, $R_= =\{G|\gfr=0\}$.}
\item{$L_\geq=L_>\cup L_=$, $L_\leq = L_<\cup L_=$, $R_\geq=R_>\cup R_=$, $L_\leq = R_<\cup R_=$.}
\end{definition}

Next we can use this to give the definition of outcome classes for scoring play games.  Note that scoring play games, unlike normal and mis\`ere play games have five outcome classes.

\begin{definition}
The outcome classes of scoring games are defined as follows:

\begin{itemize}
\item{$\lef=(L_>\cap R_>)\cup(L_>\cap R_=)\cup(L_=\cap R_>)$}
\item{$\ri=(L_<\cap R_<)\cup(L_<\cap R_=)\cup(L_=\cap R_<)$}
\item{$\n=L_>\cap R_<$}
\item{$\pre=L_<\cap R_>$}
\item{$\ti =L_=\cap R_=$}
\end{itemize}
\end{definition}

\subsection{An Example}

In this section we will give an example of Pirates and Treasure, so that the reader has a better idea for how this game is played.  Consider the game shown in figure \ref{ex}, this figure shows a typical Pirates and Treasure position, as well as, a sequence of moves that the players could make.  

In the diagrams $L$ represents Left's current position, $R$ Right's current position, and the numbers represent the amount of treasure at that vertex.  The number in brackets will represent the current score, we also change numbered nodes to non-numbered nodes once they have been visited.  This indicates that the pirate has gathered all of the available treasure at that particular node.

\begin{figure}[htb]
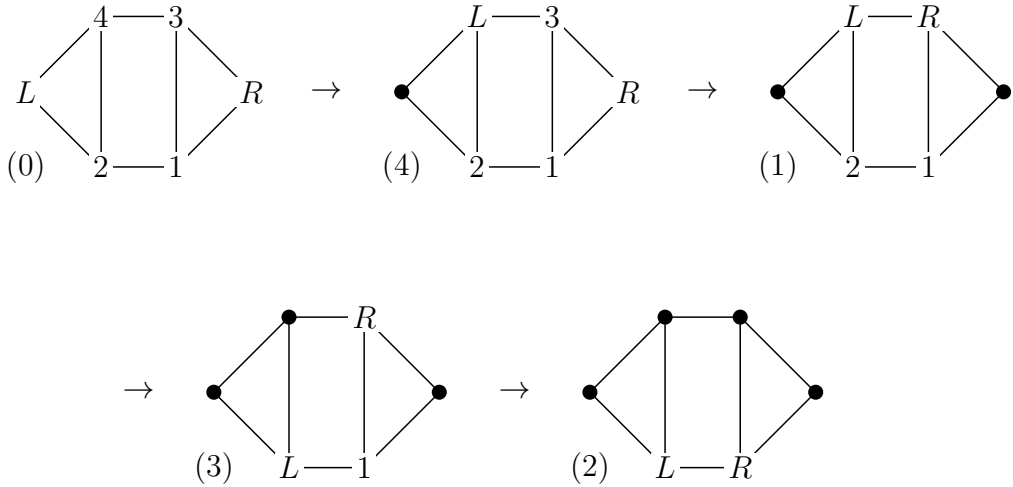

\begin{center}
\begin{graph}(13,6)

\roundnode{1}(0,5)\roundnode{2}(1,6)\roundnode{3}(1,4)
\roundnode{4}(2,6)\roundnode{5}(2,4)\roundnode{6}(3,5)

\edge{1}{2}\edge{1}{3}\edge{2}{4}\edge{3}{5}
\edge{4}{6}\edge{5}{6}\edge{2}{3}\edge{4}{5}

\nodetext{1}{$L$}\nodetext{2}{4}\nodetext{3}{2}
\nodetext{4}{3}\nodetext{5}{1}\nodetext{6}{$R$}

\freetext(0,4){$(0)$}

\roundnode{7}(5,5)\roundnode{8}(6,6)\roundnode{9}(6,4)
\roundnode{10}(7,6)\roundnode{11}(7,4)\roundnode{12}(8,5)

\edge{7}{8}\edge{7}{9}\edge{8}{10}\edge{9}{11}
\edge{10}{12}\edge{11}{12}\edge{8}{9}\edge{10}{11}

\nodetext{8}{$L$}\nodetext{9}{2}
\nodetext{10}{3}\nodetext{11}{1}\nodetext{12}{$R$}

\freetext(5,4){$(4)$}

\roundnode{13}(10,5)\roundnode{14}(11,6)\roundnode{15}(11,4)
\roundnode{16}(12,6)\roundnode{17}(12,4)\roundnode{18}(13,5)

\edge{13}{14}\edge{13}{15}\edge{14}{16}\edge{15}{17}
\edge{16}{18}\edge{17}{18}\edge{14}{15}\edge{16}{17}

\nodetext{14}{$L$}\nodetext{15}{2}
\nodetext{16}{$R$}\nodetext{17}{1}

\freetext(10,4){$(1)$}

\roundnode{19}(2.5,1)\roundnode{20}(3.5,2)\roundnode{21}(3.5,0)
\roundnode{22}(4.5,2)\roundnode{23}(4.5,0)\roundnode{24}(5.5,1)

\edge{19}{20}\edge{19}{21}\edge{20}{22}\edge{21}{23}
\edge{22}{24}\edge{23}{24}\edge{20}{21}\edge{22}{23}

\nodetext{21}{$L$}
\nodetext{22}{$R$}\nodetext{23}{1}

\freetext(2.5,0){$(3)$}

\roundnode{25}(7.5,1)\roundnode{26}(8.5,2)\roundnode{27}(8.5,0)
\roundnode{28}(9.5,2)\roundnode{29}(9.5,0)\roundnode{30}(10.5,1)

\edge{25}{26}\edge{25}{27}\edge{26}{28}\edge{27}{29}
\edge{28}{30}\edge{29}{30}\edge{26}{27}\edge{28}{29}

\nodetext{27}{$L$}
\nodetext{29}{$R$}

\freetext(7.5,0){$(2)$}

\freetext(4,5){$\rightarrow$}\freetext(9,5){$\rightarrow$}\freetext(1.5,1){$\rightarrow$}\freetext(6.5,1){$\rightarrow$}

\end{graph}
\end{center}
\caption{An example of Pirates and Treasure}\label{ex}
\end{figure}

In  this particular example, $G^{SL}_F=G^{SR}_F=2$.  Note that the winner is not related to who moves last.  If Left moves first then Right moves last, but Left wins.  If Right moves first then Left moves last, but again, Left wins.

\section{Complexity}

As always when studying a new game --or problem-- like this, the very first question we ask as a matter-of-course is ``how hard is it?''.  This is a very important question, and as we well show it is NP-hard to determine the final score of this game, and it remains NP-hard for various types of graph.

\vspace{0.5cm}

\noindent Problem: \textbf{Hamiltonian Path}

\vspace{0.5cm}

\noindent Instance: A Graph $G=(V, E)$.

\vspace{0.5cm}

\noindent Question: Does $G$ contain a Hamiltonian Path?

\vspace{0.5cm}

This problem has been shown to be NP-complete \cite{rk}.  We will be doing a reduction from Hamiltonian Path to our problem.

We define our problem as follows:

\vspace{0.5cm}

\noindent Problem: \textbf{Pirates and Treasure}

\vspace{0.5cm}

\noindent Instance: A Graph $G=(V, E)$, weight $w(v)\in \mathbb{Z}^+$ for each $v\in V$, specified vertices $L$ and $R$.

\vspace{0.5cm}

\noindent Question: Can Left win moving first on $G$?

\begin{theorem}\label{red}
Pirates and treasure is NP-hard.
\end{theorem}

\begin{proof}  To do this reduction we first start with a graph $G$ and pick a vertex $L$ on $G$.  The vertex $L$ will be the starting position of player Left.  We then give all vertices on $G$ value 1.

We add a path $P$ to $G$, such that $L$ is one of the end vertices.  We then choose the vertex $R$ on $P$ that is adjacent to $L$, and let it be the starting position of player Right.  $P$ is chosen such that $|P|=|V|$.  The reason we choose it this way is to ensure that Left can win only by visiting every vertex of the graph $G$.  That is, Right has $|V|-2$ (subtracting the vertices $L$ and $R$) vertices he can visit, while Left has at most $|V|-1$ (subtracting the vertex $L$).  

By the rules of our game, Left cannot move on $P$, since the vertex $R$ was once occupied by Right.  Likewise, Right cannot move onto $G$ since he must move to vertex $L$, and this was once occupied by Left.

If there is no hamiltonian path then Left can only visit at most $|V|-2$ of the vertices on $G$, meaning that Right is guaranteed to tie.  If there is a hamiltonian path, then Left will make the final move of the game and visit all $|V|-1$ vertices, while Right will only have visited $|V|-2$ vertices.

Therefore, Left can win moving first, if and only if, there is a hamiltonian path on $G$ and the theorem is proven.

\begin{figure}[htb]
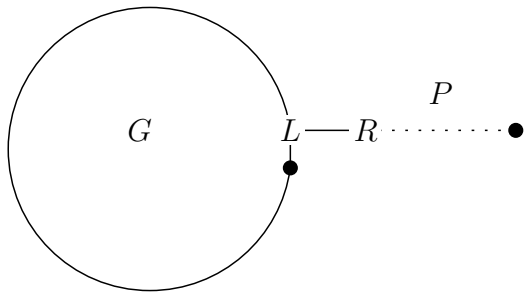

\begin{center}
\begin{graph}(7,4)

\roundnode{1}(4,2)\roundnode{4}(4,1.5)\roundnode{2}(5,2)\roundnode{3}(7,2)

\bow{1}{4}{-7.5}\edge{1}{4}\edge{1}{2}\edge{2}{3}[\graphlinedash{1 4}]

\nodetext{1}{$L$}\nodetext{2}{$R$}

\freetext(2,2){$G$}\freetext(6,2.5){$P$}

\end{graph}
\end{center}
\caption{Reduction from Hamiltonian Path.}
\end{figure}

\end{proof}

The problem Hamiltonian path remains NP-complete if $G$ is planar, cubic, 3-connected, or has no face with fewer than 5 edges \cite{gjt}.  It also remains NP-complete if $G$ is bipartite \cite{mk}, or a grid graph \cite{ips}.

\begin{definition}
A grid graph is the graph whose vertices correspond to the points in the plane with integer coordinates, $x$-coordinates being in the range $1,..., n$, $y$-coordinates being in the range $1,..., m$, and two vertices are connected by an edge whenever the corresponding points are at distance 1.
\end{definition}

\begin{theorem}
Pirates and Treasure is NP-hard if $G$ is either a planar or, a grid graph.
\end{theorem}

\begin{proof}  The proof of this is almost identical to theorem \ref{red}.  Adding a path to the outer face of a planar graph, is still a planar graph.  Likewise for a grid-graph.  Therefore, we can use the same reduction method, and the theorem is proven.
\end{proof}

For the remaining types of graph, we must use a different proof technique.  The reason is that if you add a path, to say, a cubic graph, then the resulting graph is no longer cubic.  If we want to say that the game remains NP-hard for cubic graphs, say, then the graph we use for our reduction must also be cubic.  So we make the following conjecture.

\begin{conjecture}
Pirates and Treasure is NP-hard if $G$ is either cubic, or 3-connected.
\end{conjecture}

Since it is unlikely that this game is in NP, and this game is clearly in PSpace, we also make this conjecture.

\begin{conjecture}
Pirates and Treasure is PSpace-complete.
\end{conjecture}

\section{The Game in General}

Due to the fact that scoring games are not as nicely behaved as normal play games, we must devise a new technique for studying them.  The technique that we propose is to restrict our set to only include those scoring games that represent a position of the game we wish to analyse.

In some cases this makes the problem much simpler, as was shown by Fraser Stewart in his paper on impartial scoring play games \cite{fs2}.  In this paper he looked only at the set of impartial games, and in doing was able to devise a general strategy for solving any scoring play octal game.  Here we will attempt to do the same thing, but for Pirates and Treasure.

The most obvious question to ask is, ``will playing greedily always work?''.  The answer to that is ``no'', as demonstrated in the following example.

\begin{example}  Consider the game in figure \ref{ex1}.

\begin{figure}[htb]
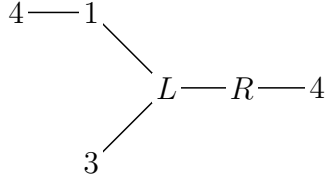

\begin{center}
\begin{graph}(4,2)

\roundnode{1}(0,2)\roundnode{2}(1,2)\roundnode{3}(2,1)
\roundnode{4}(1,0)\roundnode{5}(3,1)\roundnode{6}(4,1)

\edge{1}{2}\edge{2}{3}\edge{3}{4}
\edge{3}{5}\edge{5}{6}

\nodetext{1}{4}\nodetext{2}{1}\nodetext{3}{$L$}
\nodetext{4}{3}\nodetext{5}{$R$}\nodetext{6}{4}

\end{graph}
\end{center}
\caption{Playing greedily is not the best strategy.}\label{ex1}
\end{figure}

If Left were playing greedily he would move to the neighbour with value 3.  However, if he does so Right will move and get 4 points and therefore Left will lose.  Left's best strategy is to move to the neighbour with value 1, Right still moves and gets 4 points, but then Left can move again, get 4 points and win.  So playing greedily certainly does not always work.

\end{example}

\begin{definition}
$\mathcal{PT}=\{G|G\hbox{ represents a position of Pirates and Treasure}\}$
\end{definition}

\begin{theorem}
For all $G\in\mathcal{PT}$, if $G\not\cong 0$ then $G\neq 0$.
\end{theorem}

\begin{proof}  First let $P$ be a single edge, with one Right piece on it (note that $P=\{.|a|b\}$), and let Left move first on $G\plus P$, where $G$ is any graph that has at least one Left piece.  The case where Right moves first will follow by symmetry.  Note, $P^{SL}_F=a$, and we let $a\geq 0$.

We will label the vertices of $P$, $p_1$ and $p_2$.  We place the Right piece of $p_1$ and give $p_2$ a value that is larger than the sum of all the values on $G$.  Since $G$ is a finite graph, we can always do this.

Left moving first, must move on $G$, since he has no move on $P$.  Right simply moves to $p_2$ and wins.  Therefore $(G\plus P)^{SL}_F<0$, i.e. $G\plus P\not\approx P$ and the theorem is proven.
 
\end{proof}

This means that it is highly unlikely that we will be able to find any general technique for solving different variations of this game.

\subsection{Comparison With Normal Play}

In this section we are aiming to show that it is possible to find a variation of Pirates and Treasure that behaves very similarly to a normal play game.  That is, best strategy under normal play, corresponds to best strategy under scoring play.  First, table \ref{tabn} shows the sums of the four outcome classes under normal play.

\begin{table}[htb]
\begin{center}
\begin{tabular}{c|cccc}
$G+H$&$G\in\pre$&$G\in\lef$&$G\in\ri$&$G\in\n$\\\hline
$H\in\pre$&$\pre$&$\lef$&$\ri$&$\n$\\
$H\in\lef$&$\lef$&$\lef$&$\lef$, $\ri$, $\n$, $\pre$&$\lef$, $\n$\\
$H\in\ri$&$\ri$&$\lef$, $\ri$, $\n$, $\pre$&$\ri$&$\ri$, $\n$\\
$H\in\n$&$\n$&$\lef$, $\n$&$\ri$, $\n$&$\lef$, $\ri$, $\n$, $\pre$\\
\end{tabular}
\end{center}
\caption{Outcome Class Table for Normal Play Games}\label{tabn}
\end{table}

Our aim, is to find a variation of Pirates and Treasure that gives an outcome class which looks like that one.  By doing this, we are effectively demonstrating that there is a non-trivial subset of scoring play games that exhibit the same ``nice'' behaviour of normal play games.

First we define the games $\mathcal{PT}_x$.

\begin{definition}
$\mathcal{PT}_x$, is a subset of $\mathcal{PT}$, where every node on the graph has value $x\in\mathbb{R}$, $x>0$, and for all $G\in\mathcal{PT}_x$, $G^S=0$.
\end{definition}

What the following theorems will show is that $\mathcal{PT}_x$ exhibits behaviour that is almost identical to normal play.  In all diagrams $R$ represents a Right piece, and $L$ represents a Left piece.

\begin{theorem}\label{out}
If $G\in\mathcal{PT}_x$, then $G$ belongs to either $\lef$, $\ri$, $\n$ or $\ti$, i.e. there are no $\pre$ positions.
\end{theorem}

\begin{proof}  Whenever a player moves they gain $x$ points.  So the first player to move will have an $x$ point advantage over the second player.  Since all nodes have value $x$, the most the second player can do is bring the game back to a tie.  Therefore when the game ends, either the first player has won, or the game is a tie.  

For a position $G$ to be in $\pre$, the second player has to be able to win outright.  But, this is impossible, therefore there are no $\pre$ positions, and the theorem is proven.
\end{proof}

To show the similarities between this particular variation, and normal play, we will be looking at the outcome class table.  This is given in the following theorem.

\begin{theorem}\label{tab}  The outcome class table for $\mathcal{PT}_x$ is given as follows:
\begin{table}[htb]
\begin{center}
\begin{tabular}{c|cccc}
$G\plus H$&$G\in\ti$&$G\in\lef$&$G\in\ri$&$G\in\n$\\\hline
$H\in\ti$&$\ti$&$\lef$&$\ri$&$\n$\\
$H\in\lef$&$\lef$&$\lef$&$\lef$, $\ri$, $\n$, $\ti$&$\lef$, $\n$\\
$H\in\ri$&$\ri$&$\lef$, $\ri$, $\n$, $\ti$&$\ri$&$\ri$, $\n$\\
$H\in\n$&$\n$&$\lef$, $\n$&$\ri$, $\n$&$\lef$, $\ri$,$\n$, $\ti$\\
\end{tabular}
\end{center}
\caption{Outcome Class Table for $PT_x$}
\end{table}

\end{theorem}

\begin{proof}  The proof of this will be split into five cases, the remaining cases follow by symmetry.

\noindent \textbf{Case 1}: $G\in\mathcal{X}$,  $H\in\ti$ implies $G\plus H\in \mathcal{X}$, where $\mathcal{X}=\lef, \ri, \n$ or $\ti$.

Since $H\in\ti$, this implies that the second player to move, must be the last player to move.  First assume that Left wins moving first on $G$, since the case where Right wins moving first follows by symmetry.  Again, Left winning on $G$ implies that he also moves last on $G$.

So when they play $G\plus H$, Left will choose his winning move on $G$, and move to $G^L\plus H$.  If Right also moves on $G$, i.e. moves to $G^{LR}\plus H$, Left will respond on $G$.  When $G$ is over, they must play $H$, which ends in a tie, and therefore Left wins.  

If Right chooses to move to $G^L\plus H^R$, then Left will respond by moving to $G^L\plus H^{RL}$.  Since $H$ ends in a tie, we know that Left must move last on $H$, and therefore can force Right to move first on $G^L$, which he loses.  So if Left wins moving first on $G$, then Left wins moving first on $G\plus H$.

We know that neither Left nor Right can win moving second on any game $G$, therefore the final case to consider is Left ties moving second on $G$.  Left can tie $G\plus H$, by simply following Right's moves, i.e.  Left moves on the same component as Right.

Since Right cannot win moving first or either $G$ or $H$, Right will choose to move to $G^R\plus H$, otherwise he may give Left an opportunity to win.  If Left chooses to move to $G^R\plus H^L$, then Right can still force a tie by playing to $G^R\plus H^{LR}$.  In other words, Left cannot change the parity of $G\plus H$, and therefore the best he can do is tie moving second on $G\plus H$.

\noindent \textbf{Case 2}: $G$ and $H\in\lef$ implies $G\plus H\in\lef$.

Left playing first on $G\plus H$, simply makes his winning move on $G$ or $H$.  Then, whichever component Right chooses to move on, Left will also move on.  Since we know that Left can move last on both $G$ and $H$, he is guaranteed to keep his advantage over Right and therefore win, moving first.

When Right moves first, Left can at least tie by playing the same strategy as before (i.e. moving on the same component as Right).  Therefore $(G\plus H)^{SL}_F>0$ and $(G\plus H)^{SR}_F\geq 0$, and we conclude that $G\plus H\in\lef$.

\noindent \textbf{Case 3}: $G\in \lef$, $H\in\ri$ implies $G\plus H\in\mathcal{X}$, where $\mathcal{X}=\lef, \ri, \n$ or $\ti$.

\begin{center}
\begin{graph}(3,4)

\roundnode{1}(0,0)\roundnode{2}(0,1)\roundnode{3}(0,2)\roundnode{4}(0,3)\roundnode{5}(2,0)
\roundnode{6}(2,1)\roundnode{7}(2,2)\roundnode{8}(2,3)

\edge{1}{2}\edge{2}{3}\edge{3}{4}\edge{5}{6}
\edge{6}{7}\edge{7}{8}

\nodetext{1}{$x$}\nodetext{2}{$L$}\nodetext{3}{$x$}\nodetext{4}{$R$}\nodetext{5}{$x$}
\nodetext{6}{$R$}\nodetext{7}{$x$}\nodetext{8}{$L$}

\freetext(1,1.5){$\plus$}\freetext(3,1.5){$\in\ti$}

\end{graph}
\end{center}

\begin{center}
\begin{graph}(3,4)

\roundnode{1}(0,0)\roundnode{2}(0,1)\roundnode{3}(0,2)\roundnode{4}(0,3)\roundnode{5}(2,0)
\roundnode{6}(2,1)

\edge{1}{2}\edge{2}{3}\edge{3}{4}\edge{5}{6}

\nodetext{1}{$x$}\nodetext{2}{$L$}\nodetext{3}{$x$}\nodetext{4}{$R$}\nodetext{5}{$R$}
\nodetext{6}{$x$}

\freetext(1,1.5){$\plus$}\freetext(3,1.5){$\in\ri$}

\end{graph}
\end{center}

\begin{center}
\begin{graph}(5,3)

\roundnode{1}(1,0)\roundnode{2}(1,1)\roundnode{3}(1,2)\roundnode{4}(0,2)
\roundnode{5}(2,2)\roundnode{6}(4,0)\roundnode{7}(4,1)

\edge{1}{2}\edge{2}{3}\edge{3}{4}\edge{3}{5}\edge{6}{7}

\nodetext{1}{$L$}\nodetext{2}{$x$}\nodetext{3}{$x$}\nodetext{4}{$R$}
\nodetext{5}{$L$}\nodetext{6}{$R$}\nodetext{7}{$x$}

\freetext(2.5,1){$\plus$}\freetext(5,1){$\in\n$}

\end{graph}
\end{center}

The case $G\in\lef$, $H\in\ri$ implies $G\plus H\in\lef$ follows by symmetry.

\noindent \textbf{Case 4}: $G\in \lef$, $H\in\n$ implies $G\plus H\in\lef$ or $\n$.

If $G\in\lef$ and $H\in \n$, then Left moving first on $G\plus H$, can win by moving $G\plus H^L$.  The reason is that he will gain a 1 point advantage over Right, and whichever component Right moves on, Left will move on.  Since he moves last on both $G$ and $H$, playing this strategy will guarantee that he maintains his 1 point advantage over Right.

Since Left can always win moving first on $G\plus H$, whenever $G\in\lef$ and $H\in\n$, then we conclude that $G\plus H$ cannot be in $\ri$ or $\ti$.  To complete the proof, we simply give an example of $G\plus H\in\lef$ and $G\plus H\in\n$.

\begin{center}
\begin{graph}(3,3)

\roundnode{1}(2,0)\roundnode{2}(2,1)\roundnode{3}(2,2)
\roundnode{4}(0,0)\roundnode{5}(0,1)

\edge{1}{2}\edge{2}{3}\edge{4}{5}

\nodetext{1}{$L$}\nodetext{2}{$x$}\nodetext{3}{$R$}
\nodetext{4}{$L$}\nodetext{5}{$x$}

\freetext(1,1){$\plus$}\freetext(3,1){$\in\lef$}

\end{graph}
\end{center}

\begin{center}
\begin{graph}(5,4)

\roundnode{1}(0,0)\roundnode{2}(0,1)\roundnode{3}(0,2)\roundnode{4}(0,3)
\roundnode{5}(3,1)\roundnode{6}(3,2)\roundnode{7}(2,2)\roundnode{8}(4,2)

\edge{1}{2}\edge{2}{3}\edge{3}{4}
\edge{5}{6}\edge{6}{7}\edge{6}{8}

\nodetext{1}{$x$}\nodetext{2}{$L$}\nodetext{3}{$x$}\nodetext{4}{$R$}
\nodetext{5}{$x$}\nodetext{6}{$x$}\nodetext{7}{$L$}\nodetext{8}{$R$}

\freetext(1,1.5){$\plus$}\freetext(5,1.5){$\in\n$}

\end{graph}
\end{center}

\noindent \textbf{Case 5}: $G$ and $H\in\n$ implies $G\plus H\in\mathcal{X}$, where $\mathcal{X}=\lef, \ri, \n$ or $\ti$.

\begin{center}
\begin{graph}(5,3)

\roundnode{1}(1,0)\roundnode{2}(1,1)\roundnode{3}(1,2)\roundnode{4}(0,2)\roundnode{5}(2,2)
\roundnode{6}(2,0)\roundnode{7}(2,1)\roundnode{8}(4,0)\roundnode{9}(4,1)\roundnode{10}(4,2)

\edge{1}{2}\edge{2}{3}\edge{3}{4}\edge{3}{5}
\edge{6}{7}\edge{8}{9}\edge{9}{10}

\nodetext{1}{$L$}\nodetext{2}{$x$}\nodetext{3}{$x$}\nodetext{4}{$R$}\nodetext{5}{$L$}
\nodetext{6}{$R$}\nodetext{7}{$x$}\nodetext{8}{$L$}\nodetext{9}{$x$}\nodetext{10}{$R$}

\freetext(3,1){$\plus$}\freetext(5,1){$\in\ri$}

\end{graph}
\end{center}

\begin{center}
\begin{graph}(3,3)

\roundnode{1}(0,0)\roundnode{2}(0,1)\roundnode{3}(0,2)
\roundnode{4}(2,0)\roundnode{5}(2,1)\roundnode{6}(2,2)

\edge{1}{2}\edge{2}{3}\edge{4}{5}\edge{5}{6}

\nodetext{1}{$L$}\nodetext{2}{$x$}\nodetext{3}{$R$}
\nodetext{4}{$L$}\nodetext{5}{$x$}\nodetext{6}{$R$}

\freetext(1,1){$\plus$}\freetext(3,1){$\in\ti$}

\end{graph}
\end{center}

\begin{center}
\begin{graph}(7,3)

\roundnode{1}(1,0)\roundnode{2}(1,1)\roundnode{3}(1,2)\roundnode{4}(0,2)\roundnode{5}(2,2)
\roundnode{6}(5,1)\roundnode{7}(5,2)\roundnode{8}(4,2)\roundnode{9}(6,2)

\edge{1}{2}\edge{2}{3}\edge{3}{4}\edge{3}{5}
\edge{6}{7}\edge{7}{8}\edge{7}{9}

\nodetext{1}{$x$}\nodetext{2}{$x$}\nodetext{3}{$x$}\nodetext{4}{$L$}\nodetext{5}{$R$}
\nodetext{6}{$x$}\nodetext{7}{$x$}\nodetext{8}{$L$}\nodetext{9}{$R$}

\freetext(3,1){$\plus$}\freetext(7,1){$\in\n$}

\end{graph}
\end{center}

The case $G\in\lef$, $H\in\ri$ implies $G\plus H\in\lef$ follows by symmetry.  All remaining cases follow by symmetry, and therefore the theorem is proven.

\end{proof}

\begin{theorem}
$G\plus (-G)\in\ti$ for all $G\in \mathcal{PT}_x$.
\end{theorem}

\begin{proof}    To prove this, first consider Left moving first, since Right moving first will follow by symmetry.  Right can tie $G\plus (-G)$, simply by playing the ``tweedle-dum, tweedle-dee'' strategy, i.e. whichever move Left makes, Right makes the identical move in the opposite component.  Since this will allow Right to move last on $G\plus (-G)$, and therefore tie the game.

We know that Right cannot win moving second, because from theorem \ref{out}, there are no $\pre$ positions.  Therefore Right's strategy can guarantee him a tie, and $G\plus (-G)\in\ti$.  The theorem is proven.

\end{proof}

The natural question to ask is, ``are the sets $\mathcal{PT}_x$ groups?''.  The answer to that is ``no''.  The reason is that the games $G^L$ are $G^R$ are not in the set.  We demand that every game $G\in\mathcal{PT}_x$ have $G^S=0$.  But, $g^{LS}=x$ and $g^{RS}=-x$, for all $g^L\in G^L$ and $g^R\in G^R$.  So when a player moves, he is moving to something outside of the set.

However, what we have shown is that this particular variant of Pirates and Treasures behaves very similarly to a normal play game.  In fact the similarity goes further than that.  What we will attempt to convince the reader, although we have no proof of this, is that a winning a strategy for Pirates and Treasure under normal play, is identical to a ``non-losing'' strategy of a game in $\mathcal{PT}_x$.

Consider the game in figure \ref{half}.  If we played it under normal play (i.e. last player to move wins), this game has value $\{-1,0|1\}=\{0|1\}=\frac{1}{2}$.  Now consider the game in figure \ref{add}.  Under normal play Right's best move, moving first, is to move the Right piece on the left hand graph down adjacent to the Left piece.  This is exactly the same for scoring play.  Likewise, Left's best move --moving first-- is to move his piece up, so it is adjacent to the Right piece, under both normal and scoring play.

\begin{figure}[htb]
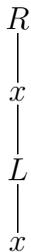

\begin{center}
\begin{graph}(0,3)

\roundnode{1}(0,0)\roundnode{2}(0,1)\roundnode{3}(0,2)\roundnode{4}(0,3)

\edge{1}{2}\edge{2}{3}\edge{3}{4}

\nodetext{1}{$x$}\nodetext{2}{$L$}\nodetext{3}{$x$}\nodetext{4}{$R$}

\end{graph}
\end{center}
\caption{One Half?}\label{half}
\end{figure}

\begin{figure}[htb]
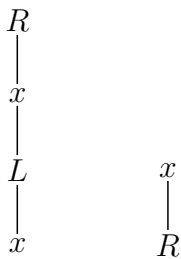

\begin{center}
\begin{graph}(2,3)

\roundnode{1}(0,0)\roundnode{2}(0,1)\roundnode{3}(0,2)\roundnode{4}(0,3)
\roundnode{5}(2,0)\roundnode{6}(2,1)

\edge{1}{2}\edge{2}{3}\edge{3}{4}
\edge{5}{6}

\nodetext{1}{$x$}\nodetext{2}{$L$}\nodetext{3}{$x$}\nodetext{4}{$R$}
\nodetext{5}{$R$}\nodetext{6}{$x$}

\end{graph}
\end{center}
\caption{An $\ri$ position, under normal and scoring play.}\label{add}
\end{figure}

In fact, if the reader is particularly vigilant, and checks the best strategies in the examples given in the proof of theorem \ref{tab}, he will find that they are exactly the same as the best strategies under normal play.  So what we have shown is that there is a non-trivial subset of scoring play games, that behaves very similarly to normal play games.

\subsubsection*{A Few Notes}

There are several things we need to note about this idea.  The first is that this will not work on the set $\mathcal{PT}_x\cup\mathcal{PT}_y$, where $x\neq y$.  It will also not work if $G^S\neq 0$, i.e. before the players have moved the score is something other than zero.  Finally, if we let $x=0$, then this game is trivial since every position is a $\ti$ position.

The reader may ask ``but didn't you just change it to last to move wins?''.  The answer to that is certainly ``no''.  We did not change the rules of the game, all we did was looked at a particular case of the game.  The real question is ``are there any other scoring games that fall into this set?''.  So we leave the following open problem.

\begin{problem}
Can you define, and classify the set of scoring play games that behave like a normal play game?  If yes, which games lie in this set?
\end{problem}

\subsection{Comparison With Mis\`ere Play}

In this paper we do not intend to say much about a comparison with mis\`ere play.  All we really be doing is giving examples to show that there are variations where the winning strategy under scoring play is identical to the winning strategy under mis\`ere play.   

To show that it is very similar to mis\`ere play is much harder, given that the general structure of mis\`ere games is not as ``nice'' as normal play.  All we will be doing is demonstrating that the last player to move does not win, i.e. he loses or ties.  So, therefore both players are trying \emph{not} to move last, just like a mis\`ere play game.

\begin{definition}
$\mathcal{PT}_{-x}$, is a subset of $\mathcal{PT}$, where every node on the graph has value $x\in\mathbb{R}$, $x>0$, and for all $G\in\mathcal{PT}_{-x}$, $G^S=0$.
\end{definition}

\begin{theorem}\label{outm}
If $G\in\mathcal{PT}_{-x}$, then $G$ belongs to either $\lef$, $\ri$, $\pre$ or $\ti$, i.e. there are no $\n$ positions.
\end{theorem}

\begin{proof}  The proof of this is very similar to the proof given for theorem \ref{out}.  When a player moves he loses $x$ points, i.e. the second player has an $x$ point advantage.  Since every time the second player moves he brings the game back to a tie, if the second player moves last the game will end in a tie.  Otherwise, the first player loses.

For the game to have be an $\n$ position the first player must be able to win outright, but this is impossible.  Therefore, there are no $\n$ positions and the theorem is proven.
\end{proof}

\begin{conjecture}
For all $G\in\mathcal{PT}_{-x}$, if $G\not\cong 0$ then $G\neq 0$.
\end{conjecture}

The reason we make this a conjecture rather than proving it, is because a proof is very difficult.  There is a good reason for that.  In \cite{gamo}, the authors used the game $G$ in figure \ref{mis1} to prove the same theorem for mis\`ere games.

\begin{figure}[htb]
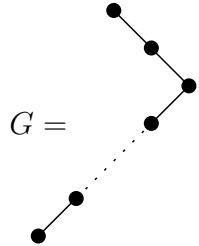

\begin{center}
\begin{graph}(2,3)

\roundnode{1}(0,0)\roundnode{2}(0.5,.5)\roundnode{3}(1.5,1.5)
\roundnode{4}(2,2)\roundnode{5}(1.5,2.5)\roundnode{6}(1,3)

\edge{1}{2}\edge{2}{3}[\graphlinedash{1 4}]\edge{3}{4}
\edge{4}{5}\edge{5}{6}

\freetext(0,1.5){$G=$}

\end{graph}
\end{center}
\caption{Game to show $G\neq 0$ under mis\`ere rules.}\label{mis1}
\end{figure}

The idea is that if we play $G+H$, where $H^L\neq\emptyset$, then Right moving first will move on $G$.  Left has to move on $H$, as he has no move on $G^R$, Right moves again on $G$.  We then let the string on Left moves on $G$ be longer than the depth of $H$, and Left will be forced to move last.

The problem is that there is no position in Pirates and Treasure, with a game tree that has that general shape.  Either a player can move from the start of the game, or he cannot move at all.  This makes proving the conjecture considerably more challenging, and it will be beyond the scope of this paper.

We will also not be examining the outcome class table for these games either.  As far as a comparison to mis\`ere play, it is known that outcome of any two games $G$ and $H$ are not related to their sum under mis\`ere play.  Therefore, there is really very little that the outcome class table would tell us.

If were were playing this game under mis\`ere rules, not every position is represented.  We know that moving last is bad, so showing that the outcome class table is also ``bad'' is not particularly interesting.  So, this is why we will not be looking at it in this paper.

To finish this section we will simply look at an example, and show that the winning strategy under scoring play corresponds to the winning strategy under mis\`ere play, and conjecture that this is always the case for these variations.  Again, this is not true if $G^S\neq 0$, or if we examine a set $\mathcal{PT}_{-x}\cup\mathcal{PT}_{-y}$, where $x\neq y$.

\begin{figure}[htb]
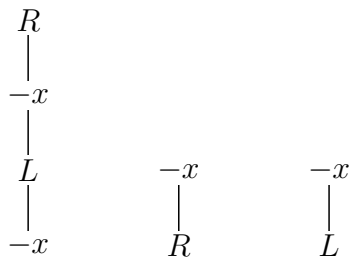

\begin{center}
\begin{graph}(4,3)

\roundnode{1}(0,0)\roundnode{2}(0,1)\roundnode{3}(0,2)\roundnode{4}(0,3)
\roundnode{5}(2,0)\roundnode{6}(2,1)\roundnode{7}(4,0)\roundnode{8}(4,1)

\edge{1}{2}\edge{2}{3}\edge{3}{4}
\edge{5}{6}\edge{7}{8}

\nodetext{1}{$-x$}\nodetext{2}{$L$}\nodetext{3}{$-x$}\nodetext{4}{$R$}
\nodetext{5}{$R$}\nodetext{6}{$-x$}\nodetext{7}{$L$}\nodetext{8}{$-x$}

\end{graph}
\end{center}
\caption{An $\ti$ position under mis\`ere play.}\label{mis}
\end{figure}

Consider the game in figure \ref{mis}.  If we played this game under mis\`ere rules, i.e. last player to move loses, then Left's best move --moving first-- would be to slide his piece, on the left hand graph, down to the lower vertex.  This is exactly the same as his best move under scoring play rules.  Likewise Right's best move, under mis\`ere rules is either to move his piece on the left hand graph down, or move his piece in the center graph up.  Again, these correspond to best strategy under scoring rules.

Also note that the difference in playing this game with negative values and positive values on the vertices, is very similar to the difference between playing it under mis\`ere and normal play rules.

We finish with one further problem:

\begin{problem}
If there is a variation of a scoring game that behaves like a normal play game, will using negative values for the points players gain or lose on their turns give a game that behaves like a mis\`ere play?  If no, can you classify which games have this property, and which do no?
\end{problem}

\section{Conclusion}

We have introduced a new game, and shown that there are variations of this game exhibits behaviour that is very similar to normal play, and variations that exhibit behaviour that is very similar to mis\`ere play.  This means that there are non-trivial subsets of scoring play games that have this same behaviour.  The question is can we define these subsets precisely, and determine which scoring games lie in them, and which do not?  We hope that we have convinced the reader that this is an interesting problem to pursue.

\end{document}